\documentclass[a4paper]{amsart}
\usepackage{amsthm,amsmath,amssymb,mathrsfs,graphicx, xcolor}
\usepackage[latin1]{inputenc}
 \newtheorem{thm}{Theorem}
 \newtheorem{prop}[thm]{Proposition}
 \newtheorem{lemma}[thm]{Lemma}
 \newtheorem{cor}[thm]{Corollary}

 \theoremstyle{definition}
 \newtheorem{definition}[thm]{Definition}
 \newtheorem{ex}[thm]{Example}

\newtheorem{const}[thm]{Construction}

 \theoremstyle{remark}
 \newtheorem{remark}[thm]{Remark}
\numberwithin{thm}{section}

\def\Spec{{\rm Spec}\,}

\def\GL{{\rm GL}}
\def\SL{{\rm SL}}

\def\Quot{{\rm Quot}}

\def\Mat{{\rm M}}

\def\rk{{\rm rk}}
\def\defect{{\rm def}}

\def\dom{{\rm dom}}
\def\codim{{\rm codim}}

\def\dim{{\rm dim}\,}

\def\defect{{\rm def}}

\def\N{{N}}%macro for the Newton poset, to be distinguished from Newton strata...

\def\m{{\mathfrak m}}
\def\Z{{\mathbb Z}}

\def\Q{{\mathbb Q}}

\def\O{{\mathcal O}}
\def\S{{\mathbb S}}
\def\a{{\mathfrak a}}

\long\def\forget#1{}

\begin{document}

\begin{title}
{Minimal Newton strata in Iwahori double cosets}
\end{title}

\author{Eva Viehmann}

\address{Technische Universit\"at M\"unchen\\Fakult\"at f\"ur Mathematik - M11 \\ Boltzmannstr. 3\\85748 Garching bei M\"unchen\\Germany}
\email{viehmann@ma.tum.de}
\thanks{The author was partially supported by ERC Consolidator Grant 770936:\ NewtonStrat.}

%\date{}

\begin{abstract}{The set of Newton strata in a given Iwahori double coset in the loop group of a reductive group $G$ is indexed by a finite subset of the set $B(G)$ of Frobenius-conjugacy classes. For unramified $G$, we show that it has a unique minimal element and determine this element. Under a regularity assumption we also compute the dimension of the corresponding Newton stratum. We derive corresponding results for affine Deligne-Lusztig varieties.}
\end{abstract}
\maketitle
\section{Introduction}\label{intro}

Let $F$ be a local field with ring of integers $\mathcal O_F$, uniformizer $t$ and residue field $\mathbb{F}_q$ of characteristic $p$. Let $L$ denote the completion of the maximal unramified extension of $F$, and ${\mathcal O_L}$ its ring of integers. Throughout the paper we assume that $F$ is of equal characteristic $p$. However, most results can then be translated literally to the arithmetic case, i.e.~to $F=\mathbb Q_q$, and to the reduction of Shimura varieties (see Section \ref{secredGunr}). In particular, our main theorem (Theorem \ref{thmmain} below) also holds in that context.

Let $G$ be an unramified reductive group over $F$, and let $K$ be reductive over $\O_F$ with $K_F=G$. Denote by $\sigma$ the Frobenius of $L$ over $F$. For $b\in G(L)$ let $$[b]_G=[b]=\{g^{-1}b\sigma(g)\mid g\in G(L)\}$$ be its $\sigma$-conjugacy class and let $B(G)=\{[b]\mid b\in G(L)\}$. We fix a Borel subgroup $B$ and a maximal torus $T\subset B$ of $G$, both defined over $\O_F$. Then the elements $[b]\in B(G)$ are classified by Kottwitz \cite{Kottwitz1} by two invariants, the Newton point $\nu_b\in X_*(T)_{\dom,\mathbb{Q}}$ and the Kottwitz point $\kappa_G(b)\in \pi_1(G)_{\Gamma}$. Here $\pi_1(G)$ is the fundamental group of $G$, i.e.~the quotient of $X_*(T)$ by the coroot lattice, and $\Gamma$ denotes the absolute Galois group of $F$. There is a partial order $\preceq$ on $B(G)$ defined by $[b]\preceq [b']$ if and only if $\kappa_G(b)=\kappa_G(b')$ and $\nu_{b'}-\nu_b$ is a non-negative rational linear combination of positive coroots. A $\sigma$-conjugacy class $[b]$ is called basic if $\nu_b$ is central in $G$. If $[b]$ is basic, it is the unique minimal element of the set of $[b']$ with $\kappa_G(b')=\kappa_G(b)$. This induces a bijection between $\pi_1(G)_{\Gamma}$ and the set of basic $[b]\in B(G)$.

Let $I\subset K$ be the Iwahori sub-group scheme whose reduction modulo $t$ is the opposite of the Borel subgroup $B$. Let $\widetilde W$ be the Iwahori-Weyl group of $G$, where for all details on the notation we refer to Section \ref{secnot}. Then $G(L)=\coprod_{x\in\widetilde W} IxI$. Here we use the same letter $x$ to denote an element of $\widetilde W$ and a chosen representative in $G(L)$. For fixed $x\in\widetilde W$ we consider the set $$B(G)_x=\{[b]\in B(G)\mid [b]\cap IxI\neq\emptyset\}$$ and for $[b]\in B(G)$ let $$\mathcal N_{x,[b]}=\mathcal N_{x,[b]}^G:= [b]\cap IxI.$$ By \cite{RR} this is the set of closed points of a locally closed (reduced) subscheme of $IxI$ which we call the Newton stratum of $[b]$ in $IxI$. In this context two very natural questions are:
\begin{itemize}
\item What is $B(G)_x$?
\item What is the codimension of $\mathcal N_{x,[b]}$ in $IxI$?
\end{itemize}
In general, still very little is known. Let us describe some partial answers that have been obtained due to the joined effort of several people. An obvious necessary condition for $[b]\in B(G)_x$ is that $\kappa_G(b)=\kappa_G(x)$. Recent results of G\"ortz, He and Nie \cite{GoertzHeNie} give a necessary and sufficient condition determining if the unique basic element of $B(G)$ satisfying $\kappa_G(b)=\kappa_G(x)$ is indeed in $B(G)_x$, compare Section \ref{secbasic}. If this is the case, it is the unique smallest element of $B(G)_x$ with respect to the order on $B(G)$. Another element of $B(G)_x$ that is of particular interest is the unique maximal element $[b_x]$ of $B(G)_x$, which coincides with the generic $\sigma$-conjugacy class in the irreducible double coset $IxI$. There are descriptions of $[b_x]$ that also give finite algorithms to compute it, in \cite{trunc1} via the partial order on $\widetilde W$ and in \cite{LizQBG} via the quantum Bruhat graph. However, none of them provides a closed formula for $[b_x]$. A complete description of $B(G)_x$ is only known in very particular cases such as for example if $x$ is of minimal length in its $\sigma$-conjugacy class (in which case $B(G)_x=\{[x]\}$, compare \cite{HeAnnals}) or if $G={\rm SL}_3$ (see \cite{LizThesis}). In general (already for $G=\SL_3$), the partially ordered set $B(G)_x$ may be non-saturated, i.e. there may be gaps in form of elements $[b_1]\preceq [b_2]\preceq [b_3]\in B(G)$ with $[b_1], [b_3]\in B(G)_x$ and $[b_2]\notin B(G)_x$.

Although the dimensions of $\mathcal N_{x,[b]}\subset IxI$ are a priori infinite, such dimensions and the codimension can be defined in a meaningful (and finite) way. Then they can be related to the dimension of the corresponding affine Deligne-Lusztig variety defined as $$X_x(b)=\{g\in G(L)/I\mid g^{-1}b\sigma(g)\in IxI\}.$$ Indeed, by \cite{HeCDM}, Theorem 2.23 we have
\begin{equation}\label{eq4}
 \dim X_x(b) = \ell(x) - \langle 2 \rho, \nu(b) \rangle - \codim(\mathcal{N}_{[b]}\subset IxI).
\end{equation}
By \cite{cordial}, Corollary 2.11, a similar statement holds for the individual irreducible components. Again, very little is known about these dimensions. Notice also that in general there are Newton strata and affine Deligne-Lusztig varieties which are not equidimensional (see \cite{GoertzHe}, 5.2 for an example for $G$ of type $A_3$ and $b=1$ basic). There is a good approximation to $\dim X_x(b)$ by the so-called virtual dimension introduced by He in \cite{HeAnnals}, 10.1 as 
\begin{equation}\label{defvd}
d_x^G(b)=d_x(b)=\frac{1}{2}\bigl(\ell(x)+\ell(\eta(x))-\defect_G b\bigr)-\langle\rho, \nu_b^G\rangle,
\end{equation}
where for notation we refer again to Section \ref{secnot}. By \cite{HeCDM}, Theorem 2.30 we have
\begin{equation*}
d^G_x(b)\geq \dim X_x(b).
\end{equation*}
If $X_x(b)\neq \emptyset$ we set
$$\Delta_x(b)=\Delta^G_x(b)=d_x^G(b)-\dim X_x(b).$$ If the basic locus in $IxI$ is non-empty, and one assumes in addition that $x$ is in the shrunken Weyl chambers, then He shows that for this basic locus, $\Delta_x(b)=0$ (compare Section \ref{secbasic}). 

Results of E.~Milicevic and the author (\cite{cordial}, Theorem 2.19) study those $x$ where $\Delta_x(b_x)=0$ for the generic $[b_x]$ in $IxI$, and call them cordial. For cordial $x$, none of the above-mentioned phenomena ($\Delta_x(b)\neq 0$ for some $[b]\in B(G)_x$, non-equidimensionality of some $X_x(b)$, or gaps in $B(G)_x$) occurs in $IxI$. Corollary \ref{corconv} below provides a partial converse to this result. Altogether it therefore seems promising to approach the above two questions using the following three steps:
\begin{itemize}
\item[I.] What are the minimal elements of $B(G)_x$, and what is the dimension of the corresponding stratum? 
\item[II.] What is the maximal element $[b_x]$ of $B(G)_x$ (in terms of a closed formula), and what is $\Delta_x(b_x)$? 
\item[III.] What are the gaps in $B(G)_x$, which strata are not equidimensional, and how are these phenomena reflected in the values of $\Delta_x$?
\end{itemize}
The main result of this work is concerned with the first of these questions.
\begin{thm}\label{thmmain}
\begin{enumerate}
\item For every $x\in\widetilde W$ the set $B(G)_x$ has a unique minimal element $[m_x]$, which is obtained as follows. Let $J$ be minimal such that $x\a$ is a $(J,w,\delta)$-alcove for some $w\in W$, where $\a$ is the base alcove corresponding to $I$. Then $w^{-1}x\delta(w)\in \widetilde W_J$. Let $M_J$ be the Levi subgroup defined by $J$ and $[b_0]_{M_J}$ the unique basic element of $B(M_J)$ with $\kappa_{M_J}(b_0)=\kappa_{M_J}(w^{-1}x\delta(w))$. Then $[m_x]=[b_0]_G$. 
\item Assume that $x$ is in the regular shrunken Weyl chambers. Then $$\dim X_x(m_x)=d_x(m_x),$$ or equivalently $\Delta_x(m_x)=0$. 
\end{enumerate}
\end{thm}
Here, $\delta$ is the automorphism of $\widetilde W$ induced by $\sigma.$ The slightly technical notions of $(J,w,\delta)$-alcove and of being regular shrunken are explained in Section \ref{secpalc} and Section \ref{secshr}, respectively. We prove the first assertion in Section \ref{secbasic}, and the second in Section \ref{sec7}. On the way, we prove several comparison results between $G$ and a Levi subgroup. 

From \cite{cordial}, Theorem 2.19 and Theorem \ref{thmmain} above we obtain the following corollary which gives another result towards Step III above, complementing the results of \cite{cordial}.

\begin{cor}\label{corconv}
Suppose that $x\in\widetilde W$ satisfies $\Delta_x^G(b_x)>\Delta_x^G(m_x)$, e.g.~if $\Delta_x^G(b_x)>0$ and $x$ is in the regular shrunken Weyl chambers. Then there is a $[b'] \in B(G)$ such that 
\begin{enumerate}
\item[(a)] $[m_x]< [b']< [b_x]$ but $[b']\notin B(G)_x$ (i.e.~$B(G)_x$ is not saturated), or
\item[(b)] $[m_x]< [b']< [b_x]$ and $[b']\in B(G)_x$, but the closure of $[b']\cap IxI$ in $IxI$ is not the union of all $[b'']\cap IxI$ for $[b'']\in B(G)_x$ with $[b'']< [b']$.\qed
\end{enumerate}
\end{cor}

\noindent\emph{Acknowledgment.} We thank the two referees for helpful remarks, in particular for suggesting a simplification of the proof of Theorem \ref{thmend}. We thank U.~G\"ortz for a helpful discussion on sign conventions in Bruhat-Tits buildings. 

\section{Notation}\label{secnot}

\subsection{}
We fix an $L$-split maximal torus $S$ defined over $\O_F$ and let $T$ be the centralizer of $S$ in $G$. Let $$\widetilde W=N_S(L)/T(L)_0$$ be the Iwahori-Weyl group where $N_S$ is the normalizer of $S$ in $G$ and where $T(L)_0\subset K(\O_L)$ is the unique parahoric subgroup of $T$. Then the Frobenius automorphism $\sigma$ of $L$ over $F$ induces an automorphism of $\widetilde W$, which we denote by $\delta$. For each $x\in \widetilde W$ we choose a representative in $N_S(L)$ which we denote by the same letter $x$. We choose $B\subset K$ a Borel subgroup defined over $\O_F$ and let $I$ be the Iwahori subgroup whose reduction modulo $t$ agrees with the reduction of the opposite of $B$. It corresponds to a $\sigma$-invariant base alcove $\a$ containing the vertex corresponding to $K$ in the apartment for $S$ of the Bruhat-Tits building of $G_L$. We use the convention that $\a$ lies in the dominant Weyl chamber, and that a translation element $\lambda\in \widetilde W$ acts by translation by $\lambda$ on the apartment for $S$. 

We have $$G(L)=\coprod_{x\in \widetilde W}IxI.$$ Let $W_a$ be the affine Weyl group (for more details on the relevant notation and further references compare \cite{GoertzHeNie}, 2). Then our choice of $I$ induces a length function and a Bruhat order on $W_a$. There is a short exact sequence $$1\rightarrow W_a\rightarrow \widetilde W\rightarrow \Omega\rightarrow 1$$ where $\Omega$ is the stabilizer of $\a$ in the Bruhat-Tits building. It identifies $\widetilde W$ with $W_a\rtimes \Omega$. We extend the length function and the partial order from $W_a$ to $\widetilde W$ by setting $\ell(x)=0$ for $x\in \Omega$ and $x\leq y$ if and only if $x$ and $y$ are of the form $x't$ and $y't$ for some $t\in \Omega$ and $x'\leq y'\in W_a$.

Let $W=N_S(L)/T(L)$ be the (finite) Weyl group of $G$. Then the natural projection $\widetilde W\rightarrow W$ has kernel $X_*(T)$. We have a splitting $W\rightarrow \widetilde W$ associated with the (hyper)special vertex of $\a$ corresponding to $K$. It induces an isomorphism $\widetilde W\cong X_*(T)\rtimes W$.

Let $\Phi$ be the set of roots of $G$ over $L$ relative to $S$. For $a\in \Phi$ we denote by $U_a$ the corresponding root subgroup of $G$. Our choice of Borel subgroup determines a basis $\S$ of $\Phi$ of simple roots, which we also identify with the set of simple reflections in $W$. Let $\Phi^+$ be the set of positive roots. For a subset $J\subseteq \S$ let $W_J$ be the subgroup of $W$ generated by the simple reflections in $J$ and let $\Phi_J\subseteq \Phi$ be the roots spanned by $J$. Let $\Phi^+_J=\Phi^+\cap\Phi_J$. Let $M_J$ be the subgroup of $G$ generated by $T$ and all $U_a$ for $a\in \Phi_J$. If $\delta(J)=J$, then $M_J$ is defined over $F$. Let $\widetilde W_J=X_*(T)\rtimes W_J$. Then $\widetilde W_J$ is the Iwahori-Weyl group of $M_J$.

For $g\in G(L)$ (or for a subset of $G(L)$) and $x\in G(L)$ we write ${}^xg=xgx^{-1}$, and we also use ${}^{\sigma}g=\sigma(g)$.

\begin{definition}[\cite{GoertzHeNie}, 3.3]
Let $J\subseteq \mathbb{S}$ with $\delta(J)=J$ and $w\in W$. Let $x\in\widetilde W$. Then $x\a$ is called a $(J,w,\delta)$-alcove if $w^{-1}x\delta(w)\in \widetilde W_J$ and for every $\alpha\in w(\Phi^+-\Phi_J^+)$ we have $U_{\alpha}\cap {}^xI\subseteq U_{\alpha}\cap I$. 

In this context, we write $I_M={}^wM_J\cap I$ (following \cite{GoertzHeNie}), and $I_J=M_J\cap I$ (which we use more frequently).
\end{definition}

\subsection{Affine Deligne-Lusztig varieties in affine flag varieties}\label{secshr} For $x\in\widetilde W$ and $b\in G(L)$ the associated affine Deligne-Lusztig set is defined as $$X_x(b)=\{g\in G(L)/I\mid g^{-1}b\sigma(g)\in IxI\}.$$ It is the set of $\overline{\mathbb F}_p$-valued points of a locally closed, reduced subvariety of the affine flag variety, called the affine Deligne-Lusztig variety.

For $\mu\in X_*(T)$ let $t^{\mu}$ be the image of the fixed uniformizer $t$ under the map $\mu:\mathbb{G}_m\rightarrow T$.  We consider a map $\eta:\widetilde W\rightarrow W$ which is defined as follows. We write $x\in \widetilde W$ as $x=vt^{\mu}w$ where $v,w\in W$ and $t^{\mu}w$ maps the base alcove to the dominant chamber. Equivalently we can require that $w\in {}^{\mu}W$, i.e.~that it is the shortest representative of its $W_{\mu}$-coset $W_{\mu}w$ where $W_{\mu}$ is the centralizer of $\mu$ in $W$. Then $\eta(x)=\delta^{-1}(w)v$.

Using this notation one can explain the definition of the virtual dimension in \eqref{defvd}. Namely, $\ell(x)$ and $\ell(\eta(x))$ are the lengths of the two elements in $\widetilde W$ resp.~in $W$, and the defect is given by $\defect_G b=\rk_F G-\rk_F J_b$ where $\rk_F$ denotes the rank of a maximal $F$-split torus and where $J_b$ is the reductive group over $F$ with $$J_b(F)=\{g\in G(L)\mid g^{-1}b\sigma(g)=b\}.$$

Previous work on affine Deligne-Lusztig varieties in affine flag varieties indicates that the theory is more accessible and more is known if one assumes that the alcove $x\a$ is sufficiently far from the walls of the Weyl chambers in the Bruhat-Tits building. To make such an assumption precise, the shrunken Weyl chambers are defined as the set of $x$ such that $U_{\alpha}\cap {}^xI\neq U_{\alpha}\cap I$ for all roots $\alpha$. An element $x$ is called regular shrunken if it is shrunken and in addition $x$ can be written as $v_1t^\mu v_2$ with $v_1,v_2\in W$ and $\langle\alpha,\mu\rangle>0$ for all $\alpha\in\Phi^+$.

\section{The arithmetic case}\label{secredGunr}

By \cite{HeAnnals}, Theorem 6.1, Corollary 6.2, and Section 6.2, the condition $X_x(b)\neq \emptyset$ and also $\dim X_x(b)$ are expressed in terms of certain class polynomials and thus only depend on the combinatorial datum $(\widetilde W\cong X_*(T)\rtimes W, \delta, x, [b])$, and are independent of the group $G$ itself. Here we use that $[b]$ can be given by a representative $y$ in $\widetilde W$. 

{\it Claim.} $d_x(b)$, and thus $\Delta_x(b)$ can be expressed in terms of this combinatorial datum.

This is clear for $\ell(x)$ and $\ell(\eta(x))$. If we assume $[b]$ to be given by a representative $y$ in $\widetilde W$, then $\nu_b=\frac{1}{n} y\delta(y)\dotsm \delta^{n-1}(y)$ where $n$ is chosen such that $y\delta(y)\dotsm \delta^{n-1}(y)\in X_*(T)$. Finally, by \cite{KottwitzNewt} (1.9.1), $\defect_G(b)=\dim \a-\dim \a^{w_{b}}$ where $w_b\in W$ is the image of $y$ in the finite Weyl group. This proves the claim.

Let us now explain the analog of our main result in the arithmetic case. Let $F$ be a local field of mixed characteristic, and let all other data be as before (w.r.t.~our new choice of $F$). Then the affine Deligne-Lusztig varieties are defined as locally closed subschemes of Zhu's Witt-vector affine Grassmannian, compare \cite{Zhu}. In particular, there is a meaningful notion of dimension for them. We claim that then Theorem \ref{thmmain} also holds in this context. Indeed, in the same way as above one proves that also in the arithmetic case the statement of Theorem \ref{thmmain} is an assertion that only depends on the combinatorial datum $(\widetilde W \cong X_*(T)\rtimes W, \delta, x,y)$ where $y$ is a representative of $[b]$ in $\widetilde W$.  But the combinatorial datum associated with the problem in the arithmetic case is also represented by a group $G$ over a local function field, and $[y]\in B(G)$ for the function field case. Thus Theorem \ref{thmmain} for the function field case implies the same assertion for the arithmetic case. Similar considerations apply to other assertions concerning non-emptiness and dimension of affine Deligne-Lusztig varieties, for example to Theorem \ref{thmend}.

The arithmetic version of Theorem \ref{thmmain} in its turn can be applied to the reduction modulo $p$ of Shimura varieties with Iwahori level structure. Here, it implies that for every KR stratum in the reduction there is a unique minimal Newton stratum intersecting it non-trivially. 

\section{Normalized $(J,w,\delta)$-alcoves}\label{secpalc}

\begin{const}\label{remtransform}
Assume that $x\a$ is a $(J,w,\delta)$-alcove. Notice that in general, ${}^wM_J$ (and hence $I_M$) need not be $\sigma$-invariant if $w\neq\delta(w)$. However, $\sigma$-conjugation with $w^{-1}$ gives an isomorphism
\begin{align*}
I_Mx\sigma(I_M)=~&(wM_Jw^{-1}\cap I)x(\sigma(w)M_J\sigma(w)^{-1}\cap I)\\
&\quad\longrightarrow (M_J\cap {}^{w^{-1}}I)(w^{-1}x\delta(w))(M_J\cap {}^{\delta(w^{-1})}I)
\end{align*}
 that also preserves the Newton stratifications. 
 
By definition of $(J,w,\delta)$-alcoves, $x\a$ is a $(J,w,\delta)$-alcove if and only if it is a $(J,wm,\delta)$-alcove for any $m\in W_J$. Replacing $w$ by the minimal length representative of $wW_J$ we obtain $M_J\cap {}^{w^{-1}}I=I_J=M_J\cap I$. Then the above isomorphism reads $$I_Mx\sigma(I_M)\overset{\cong}{\longrightarrow} I_J(w^{-1}x\delta(w))I_J.$$
Notice however that we still have that ${}^{w^{-1}}I\neq I$ in general.
\end{const}

\begin{definition}
For $x\in \widetilde W$, we call $x\a$ a normalized $(J,w,\delta)$-alcove if it is a $(J,w,\delta)$-alcove such that $w$ is of minimal length in $wW_J$, or equivalently such that $M_J\cap {}^{w^{-1}}I=I_J$. Then we also call $x$ a $(J,w,\delta)$-alcove element.

In this case let $\tilde x=w^{-1}x\delta(w)$.
\end{definition}

\begin{lemma}\label{lemmin}
Let $x\in \widetilde W$ and let $J$ be a minimal element in the set of subsets of $\mathbb S$ such that there is a $w \in W$ such that $x\a$ is a $(J,w,\delta)$-alcove. Let $w$ be such that $x\a$ is a normalized $(J,w,\delta)$-alcove. Then $\tilde x\in  \widetilde{W}_J$ is not a $(J',w',\delta)$-alcove element for $M_J$ for any $\delta$-stable proper subset $J'\subset J$ and $w'\in W_{J}$. 
\end{lemma}
Construction \ref{remtransform} shows that for every $x\in \widetilde W$, there exists a pair $(J,w)$ as in the lemma.
\begin{proof}
Assume that there is a proper subset $J'\subsetneq J$ such that $\tilde x\in \widetilde{W}_J$ is a $(J',w',\delta)$-alcove element for some $w'\in W_{J}$. We want to show that then $x\a$ is a $(J',ww',\delta)$-alcove. By the first property of $(J',w',\delta)$-alcoves, $(ww')^{-1}x\delta(ww')=(w')^{-1}\tilde x\delta(w')\in \widetilde W_{J'}$. We have to verify that for every $\alpha\in (ww')(\Phi^+-\Phi_{J'}^+)$ we have $U_{\alpha}\cap {}^xI\subseteq U_{\alpha}\cap I$. We first consider the case that $\alpha\in (ww')(\Phi^+-\Phi_J^+)=w(\Phi^+-\Phi_J^+)$, where the equality follows from $w'\in W_J$. In this case the property holds since $x\a$ is a $(J,w,\delta)$-alcove. Let now $\alpha\in (ww')(\Phi_J^+-\Phi_{J'}^+)$. Conjugating by $w$ the desired inclusion $U_{\alpha}\cap {}^xI\subseteq U_{\alpha}\cap I$ is equivalent to 
\begin{equation}\label{eqlemmin}
U_{w^{-1}\alpha}\cap {}^{w^{-1}x}I\subseteq U_{w^{-1}\alpha}\cap {}^{w^{-1}}I. 
\end{equation}
We have $w^{-1}\alpha\in w'(\Phi_J^+)\subseteq \Phi_J$. Since $x$ is normalized, we have $M_J\cap {}^{w^{-1}}I=I_J$. Applying $\sigma$ to this equality and using that the right hand side and $M_J$ are by definition $\sigma$-invariant, we also obtain $M_J\cap {}^{\delta(w)^{-1}}I=I_J$. By the first equality the right hand side of \eqref{eqlemmin} is equal to $U_{w^{-1}\alpha}\cap I$. The left hand side is equal to $U_{w^{-1}\alpha}\cap {}^{\tilde x \delta(w^{-1})}I$. From $\tilde x\in \widetilde W_J$ we obtain that this is equal to $U_{w^{-1}\alpha}\cap {}^{\tilde x \delta(w^{-1})}I_J=U_{w^{-1}\alpha}\cap {}^{\tilde x}(^{\delta(w^{-1})}I\cap M_J)=U_{w^{-1}\alpha}\cap {}^{\tilde x}I_J$. Since $w^{-1}\alpha\in w'(\Phi_J^+-\Phi_{J'}^+)$, \eqref{eqlemmin} is thus nothing but $U_{w^{-1}\alpha}\cap {}^{\tilde x}I\subseteq U_{w^{-1}\alpha}\cap I$, the second property of $\tilde x$ being a $(J',w',\delta)$-alcove element.
\end{proof}

\section{Comparing sets of $\sigma$-conjugacy classes}\label{sec5}

If $G$ is as above and $H$ a subgroup, then the inclusion $\iota$ induces a natural map $B(\iota):B(H)\rightarrow B(G)$. In general, and even if we assume that $H$ is the Levi component of a parabolic subgroup of $G$, this map is neither injective nor surjective, as can be seen in the following example.

\begin{ex}
Let $G=\GL_2$ and $H$ the diagonal torus. Consider
$$b_1=\left(\begin{array}{cc}1&0\\0&t\end{array}\right),\quad 
b_2=\left(\begin{array}{cc}0&1\\t&0\end{array}\right),\quad 
b_3=\left(\begin{array}{cc}t&0\\0&1\end{array}\right)$$

Then $b_1,b_3\in H(L)$ with $[b_1]_H\neq [b_3]_H$ because their images in $\pi_1(H)\cong \Z^2$ are different, but $b_1=sb_3\sigma(s)^{-1}$ where $s\in G(L)$ is the morphism exchanging the two basis elements. Thus $[b_1]_G=[b_3]_G$ and $B(\iota)$ is not injective.

Furthermore, $[b_2]_G$ does not have a representative in $H(L)$ because its Newton point is $(1/2,1/2)\in \Q^2$, and elements of the split torus $H$ have integral Newton points. Thus $B(\iota)$ is also not surjective.
\end{ex}

However, if $[b]_G\in B(G)$ is basic and $H$ is a standard Levi subgroup of $G$ defined over $F$, then the fiber $B(\iota)^{-1}([b]_G)$ has at most one element by \cite{GoertzHeNie}, Proposition 3.5.1. Notice that only this weaker statement is in fact shown in loc.~cit., compare also their erratum.

More can also be said if we require the classes to be comparable for the partial order. Since it seems to be of independent interest, we formulate the following theorem in greater generality than what is needed for the present purpose.
\begin{thm}\label{lemcomporder}
Let $F$ be a local field and let $H\subseteq G$ be quasi-split reductive groups over $F$.  Let $[b]_{H}\preceq [b']_{H}\in B(H)$, and let $[b]_G,[b']_{G}$ be the corresponding classes in $B(G)$.
\begin{enumerate}
\item $[b]_G\preceq [b']_{G}$.
\item If moreover $[b]_{H}\neq [b']_{H}$, then also $[b]_G\neq [b']_{G}$.
\end{enumerate}
\end{thm}
\begin{proof}
We choose maximal $F$-tori $T_H\subset T$ of $H$ and $G$ and Borel subgroups $B_H\subset B$ containing them, all defined over $F$.

From $\kappa_{H}(b)=\kappa_{H}(b')$ and functoriality of the Kottwitz map we obtain $\kappa_{G}(b)=\kappa_{G}(b')$. For (1) it is thus enough to show the following statement: Let $\nu,\nu'\in X_*(T)_{\Q}$ be $H$-dominant and $\nu'-\nu$ a non-negative rational linear combination of positive coroots of $H$.  Let $\nu_{\dom},\nu'_{\dom}$ be the $G$-dominant representatives in the Weyl group orbits of $\nu,\nu'$. Then $\nu'_{\dom}-\nu_{\dom}$ is a non-negative rational linear combination of positive coroots of $G$. For (2) we have to show an analogous statement replacing \emph{non-negative} by \emph{non-trivial non-negative}, by which we mean that all coefficients of the positive coroots are non-negative and at least one coefficient is non-zero. Multiplying by a suitable positive integer we may assume that $\nu,\nu'\in X_*(T)_{H-\dom}$ are integral and that the above difference is a non-negative integral linear combination of coroots of $H$. Since these assertions do not depend on any Galois action any more, we may pass to an algebraic closure of $F$ and may thus assume that $H$ and $G$ are split.
 
Let $K$ be a hyperspecial maximal compact subgroup of $G$ and $K_H=K\cap H$. The condition that $\nu'-\nu$ a non-negative rational linear combination of positive coroots of $H$ is equivalent to $K_H\nu(t)K_H\subset \overline{K_H\nu'(t)K_H}$. In particular, $K\nu_{\dom}(t)K\cap \overline{K\nu'_{\dom}(t)K}\neq \emptyset$. Since $K\nu_{\dom}(t)K$ is an orbit for the $K\times K$-action and $\overline{K\nu'_{\dom}(t)K}$ is invariant under this action, this implies $K\nu_{\dom}(t)K\subset \overline{K\nu'_{\dom}(t)K}$. This last condition is again equivalent to the condition that $\nu'_{\dom}-\nu_{\dom}$ is a non-negative integral linear combination of positive coroots of $G$, which implies (1).

For (2) we have to show that $\nu_{\dom}\neq\nu'_{\dom}$ for the $G$-dominant representatives of the Weyl group orbits of $\nu$ and $\nu'$. 

 {\it Claim.} There is a Borel subgroup $B'$ of $G$ with $B'\cap H=B_H$ and such that $\nu$ is $G$-dominant with respect to $B'$. 

To prove the claim let $M'$ be the centralizer of $\nu$, a Levi subgroup containing $T$. Let $W^{M'}$ be the subset of elements of $W$ that are shortest representatives of their $W_{M'}$-cosets. Let $w\in W^{M'}$ be such that $w(\nu)$ is dominant with respect to $B$. Let $B'=w^{-1}Bw$. Then $\nu$ is dominant with respect to $B'$. It remains to consider the intersection with $H$. Let $\alpha\in \Phi^+_H$, the set of roots in $B_H$, or equivalently in $B\cap H=B'\cap H$. Assume first that $\langle \alpha,\nu\rangle>0$. Then since $\nu$ is also dominant with respect to $B'$, the root $\alpha$ lies in $B'$. Let now $\alpha\in \Phi_H^+$ with $\langle \alpha,\nu\rangle=0$. Then $\alpha$ is a positive root in $M'\cap H$. Since $w\in W^{M'}$, the conjugate $w(\alpha)$ is again a positive root, hence in $B$. In particular, $\alpha$ is a root in $B'$, which finishes the proof of the claim.

We replace $B$ by $B'$ and may thus assume that $\nu$ is dominant. As usual let $\rho$ denote the half-sum of the positive roots of $G$, which agrees with the sum of all fundamental weights. Recall that $\nu'-\nu\neq 0$ is a non-trivial non-negative linear combination of positive coroots of $H$ and thus also of $G$. Using the second interpretation for $\rho$ we see that $\langle\rho, \nu'\rangle>\langle\rho,\nu\rangle$. Using the first interpretation we obtain that 
\begin{eqnarray*}
\langle 2\rho, \nu'_{\dom}\rangle&=&\sum_{\alpha\in\Phi^+}\langle\alpha,\nu'_{\dom}\rangle=\frac{1}{2}\sum_{\alpha\in\Phi}|\langle\alpha,\nu'_{\dom}\rangle|=\frac{1}{2}\sum_{\alpha\in\Phi}|\langle\alpha,\nu'\rangle|\\
&=&\sum_{\alpha\in\Phi^+}|\langle\alpha,\nu'\rangle|\geq\sum_{\alpha\in\Phi^+}\langle\alpha,\nu'\rangle=\langle 2\rho, \nu'\rangle.
\end{eqnarray*}
Altogether we find $\langle\rho, \nu'_{\dom}\rangle>\langle\rho,\nu\rangle$, thus in particular $\nu'_{\dom}\neq \nu$.
\end{proof}

From \cite{GoertzHeNie}, Theorem 3.3.1 we obtain
\begin{cor}\label{cor331}
Let $x\in\widetilde W$ be a normalized $(J,w,\delta)$-alcove element. Then every element of $IxI$ is $I$-$\sigma$-conjugate to an element of $I_Mx\sigma(I_M)$, which is in turn $w^{-1}$-$\sigma$-conjugate to $I_{J}\tilde xI_J$. In particular, the restriction of $B(\iota)$ composed with $w$-$\sigma$-conjugation induces a surjection $B(M_J)_{\tilde x}\rightarrow B(G)_x$.\qed
\end{cor} 
\begin{ex}\label{exnew}
The map of the preceeding corollary is still not injective in general. For example let $G=\GL_4$ and $M\cong \GL_2\times \GL_2$ the subgroup of block diagonal matrices of size $(2,2)$. We choose $B$ to be the Borel subgroup of lower triangular matrices. Let $\tilde x\in \widetilde W_{J}$ be such that the image in the affine Weyl group of each factor $\GL_2$ is of the form $x_0=t^{(1,0)}s$ where $s$ is the non-trivial element in $W_{\GL_2}=S_2$. Let $\mu=(0,1)$. Then $x_0\in W_{\GL_2}t^{\mu}W_{\GL_2}$ implies that every $[b]\in B(\GL_2)_{x_0}$ satisfies $[b]\leq [\mu]$. We have $x_0\delta(x_0)=t^{(1,1)}$, hence $\nu_{x_0}=(\tfrac{1}{2},\tfrac{1}{2})$. Since $\mu\leq x_0$, the formula for the generic $\sigma$-conjugacy class in $Ix_0 I$ in \cite{trunc1} shows that it is equal to $[\mu]$. Thus $B(\GL_2)_{x_0}=\{[x_0],[\mu]\}$. Let $w=(14)\in W_G$. Then one can check  that the corresponding element $x=w\tilde x\delta(w^{-1})$ is a $(J=(s_1,s_3),w,\delta)$-alcove element with $B(M)_{\tilde x}$ consisting of all four pairs of elements of $B(\GL_2)_{x_0}$, and $B(G)_x$ consisting of the three elements that are obtained by identifying $([x_0],[\mu])$ and $([\mu],[x_0])$. 
\end{ex}

\section{The basic locus and the proof of Theorem \ref{thmmain}(1)}\label{secbasic}
Let $x\in \widetilde W$. Then there is a unique basic element $[b_{0,x}]\in B(G)$ with $\kappa_G(b_{0,x})=\kappa_G(x)$. It agrees with the unique minimal element of $B(G)$ mapping to $\kappa_G(x)$ under $\kappa_G$. If $\mathcal N_{x,[b_{0,x}]}\neq \emptyset$, it is therefore the unique minimal element of $B(G)_x$. By Theorem A of \cite{GoertzHeNie},  $\mathcal N_{x,[b_{0,x}]}= \emptyset$ if and only if there is a pair $(J,w)$ such that $x\a$ is a $(J,w,\delta)$-alcove and $\kappa_{M_J}(w^{-1}x\delta(w))\notin \kappa_{M_J}([b_{0,x}]\cap M_J(L))$.

In \cite{HeAnnals}, Corollary 12.2 (see also \cite{HeSurvey}, Theorem 5.3 and the corresponding footnote for the generalization to non-split $G$), He proves that if $x$ is in the shrunken Weyl chambers and $\mathcal N_{x,[b_{0,x}]}\neq \emptyset$, then 
\begin{equation}\label{eqbasic}
\dim X_x(b_{0,x})=d_x(b_{0,x}).
\end{equation}

Thus if $[b_{0,x}]\in B(G)_x$, all assertions of Theorem \ref{thmmain} are shown. We now return to the general case.

\begin{proof}[Proof of Theorem \ref{thmmain}(1)]
We choose $J,w$ as in the theorem. We may further choose $w$ such that $x\a$ is a normalized $(J,w,\delta)$-alcove, compare Construction \ref{remtransform}. By Corollary \ref{cor331}, $B(G)_x$ is the image of $B(M_J)_{\tilde x}$ in $B(G)$. By Lemma \ref{lemmin} and the results we just recalled, we obtain that $B(M_J)_{\tilde x}$ has a unique minimal element for the partial order on $B(M_J)$, which is the unique basic element $[b_{0,\tilde x}]$ of $B(M_J)$ with $\kappa_{M_J}([b_{0,\tilde x}])=\kappa_{M_J}(\tilde x)$. By Theorem \ref{lemcomporder} its image in $B(G)$ is the unique minimal element of $B(G)_x$.
\end{proof}

\section{Dimension}\label{sec7}

Let $[b]\in B(G)_x$, and $\nu_b^G$ its $G$-dominant Newton point. Recall that
\begin{eqnarray}
\nonumber\Delta^G_x(b)&:=&d_x^G(b)-\dim X_x(b)\\
\label{eq5}&=&\frac{1}{2}(\ell(x)+\ell(\eta(x))-\defect_G b)-\langle\rho, \nu_b^G\rangle-\dim X_x(b).
\end{eqnarray}

\begin{prop}\label{propprop(1)}
Let $x\a$ be a normalized $(J,w,\delta)$-alcove.
\begin{enumerate}
\item Let $b\in I_J\tilde xI_J$  and denote its $M_J$-dominant Newton point by $\nu_b$. Then
$$\ell_G(x)-\ell_{M_J}(\tilde x)=\langle 2(\rho_G-\rho_{M_J}),\nu_b\rangle.$$
\item If $b\in I_J\tilde xI_J$ is basic in $M_J$ then $\nu_b$ is $G$-dominant.
\end{enumerate}
\end{prop}
Notice that in general $\nu_b$ need not be $G$-dominant, compare Example \ref{exnew}.

\begin{proof}
By the defining property of $(J,w,\delta)$-alcoves we have that
\begin{align*}
\ell_G(x)&=\dim I/I\cap xIx^{-1}\\
&=\sum_{\alpha\in w(\Phi_J)}\dim (U_{\alpha}\cap I)/(U_{\alpha}\cap I\cap {}^xI)+\sum_{\alpha\in w(\Phi^+-\Phi_J^+)}\dim (U_{\alpha}\cap I/U_{\alpha}\cap {}^xI).
\end{align*}
We denote the two sums by $S_1$ and $S_2$, and first consider $S_1$.

Since $x\a$ is a normalized $(J,w,\delta)$-alcove, we have $M_J\cap I=M_J\cap {}^{w^{-1}}I$. Using the first description, this group is $\sigma$-invariant, thus we also have $M_J\cap I=M_J\cap {}^{\delta(w)^{-1}}I$. We conjugate this equality by $\tilde x=w^{-1}x\delta(w)\in M_J$, and obtain
$$M_J\cap {}^{\tilde x}I={}^{\tilde x}(M_J\cap I)={}^{\tilde x}(M_J\cap{}^{\delta(w)^{-1}}I)=M_J\cap {}^{\tilde x\delta(w)^{-1}}I=M_J\cap {}^{w^{-1}x}I.$$ 
We have 
\begin{align*}
S_1&=\dim ({}^wM_J\cap I)/ ({}^wM_J\cap I\cap{}^xI)=\dim (M_J\cap {}^{w^{-1}}I)/(M_J\cap {}^{w^{-1}}I\cap{}^{w^{-1}x}I)\\
\intertext{and using the above equalities this is}
&=\dim (M_J\cap I)/(\dim M_J\cap I\cap{}^{\tilde x}I)=\ell_{M_J}(\tilde x).
\end{align*}
For the first assertion it remains to show that $S_2=\langle 2(\rho_G-\rho_{M_J}),\nu_b\rangle.$ 

{\it Claim.} Let  $\mu\in X_*(T)_{M_J-\dom}$ with $\tilde x\in W_Jt^{\mu} W_J$. Then $\langle 2(\rho_G-\rho_{M_J}),\nu_b\rangle=\langle 2(\rho_G-\rho_{M_J}),\mu\rangle.$

From $b\in M_J(\mathcal O_L)\mu (t)M_J(\mathcal O_L)$ we obtain $\nu_b\preceq_{M_J} \mu $, i.e.~$\mu -\nu_b$ is a linear combination of coroots of $T$ in $M_J$ (in fact a non-negative linear combination of positive coroots). Let $\alpha^{\vee}$ be a coroot of $T$ in $M_J$. Then  $\langle \rho_G,\alpha^{\vee}\rangle=1$ and analogously for $\rho_{M_J}$. The claim follows.

By definition of $S_2$ we have
\begin{align*}
S_2&=\sum_{\alpha\in w(\Phi^+-\Phi_J^+)}\dim (U_{\alpha}\cap I/U_{\alpha}\cap {}^xI),\\
\intertext{and conjugating by $w^{-1}$ this is}
&=\sum_{\alpha\in \Phi^+-\Phi_J^+}\dim (U_{\alpha}\cap {}^{w^{-1}}I/U_{\alpha}\cap {}^{w^{-1}x}I)\\
&=\dim N_J\cap {}^{w^{-1}}I-\dim N_J\cap {}^{w^{-1}x}I.\\
\intertext{Since $N_J$, $I$, $I\cap N_J$, and hence also the dimension of subsets of $N_J$ are invariant under $\sigma$, this is}
&=\dim N_J\cap {}^{\delta(w^{-1})}I-\dim N_J\cap {}^{w^{-1}x\delta(w)\delta(w^{-1})}I\\
&=\dim N_J\cap {}^{\delta(w^{-1})}I-\dim N_J\cap {}^{\tilde x\delta(w^{-1})}I.\\
\intertext{We now decompose $\tilde x=w_1{}^{w_2}\mu$ with $w_1,w_2\in W_J$ and $\mu\in X_*(T)_{M_J-\dom}$. We use that $N_J={}^{w_1}N_J$, and furthermore that $N_J\cap I$ (and hence the notion of dimensions for subsets of $N_J$) is invariant under conjugation by $w_1$. We obtain}
&=\dim N_J\cap {}^{\delta(w^{-1})}I-\dim N_J\cap {}^{{}^{w_2}\mu\delta(w^{-1})}I\\
&=\sum_{\alpha\in \Phi^+-\Phi_J^+} \langle \alpha,{}^{w_2}\mu\rangle.\\
\intertext{The element $w_2\in W_J$ induces a permutation of $\Phi^+-\Phi_J^+$, thus this is}
&=\sum_{\alpha\in \Phi^+-\Phi_J^+} \langle \alpha,\mu\rangle.\\
&=\langle 2(\rho_G-\rho_{M_J}),\mu\rangle.
\end{align*}
Together with the claim this finishes the proof of the first assertion.

For the second assertion we assume that $\nu_b$ is central in $M_J$. Let $\alpha_0\in \mathbb S - J$ and let $J\cup \{\delta^ i(\alpha_0)\mid i\in\Z\}=J'\subset \mathbb S$ and $M'\supset M_J$ the corresponding Levi subgroup of $G$. We have to show that $\langle \alpha_0, \nu_b\rangle\geq 0$. Since $\langle \alpha, \nu_b^{M_J}\rangle= 0$ for all roots $\alpha$ in $M_J$ and since $\nu_b$ is $\delta$-invariant, this is equivalent to $\langle \rho_{M'}-\rho_{M_J}, \nu_b^{M_J}\rangle\geq 0$. By an analog of the calculation of $S_2$ above this expression is equal to $$\sum_{\alpha\in w(\Phi_{J'}^+-\Phi_J^+)}\dim (U_{\alpha}\cap I)-\dim (U_{\alpha}\cap {}^xI).$$ This in its turn is non-negative by the assumption that $x\a$ is a $(J,w,\delta)$-alcove.
\end{proof}

\begin{lemma}\label{lemlast}
Let $x\a$ be a normalized $(J,w,\delta)$-alcove in the regular shrunken Weyl chambers. Then
$$\ell_G(\eta_G(x))=\ell_{M_J}(\eta_{M_J}(\tilde x)).$$ 
\end{lemma}
\begin{proof}
We want to show the two equalities
\begin{equation}\label{eqlemlast0}
\ell_G(\eta_G(x))=\ell_{G}(\eta_G(\tilde x))=\ell_{M_J}(\eta_{M_J}(\tilde x)).\end{equation}
We begin with the first. Our assumption on $x$ to be regular implies that there is a unique decomposition $x=v_1t^\mu v_2$ with $v_1,v_2\in W={}^{\mu}W$ and $\mu$ dominant, and hence that $\eta_G(x)=\sigma^{-1}(v_2)v_1$. Further, the same consideration implies that $\eta_G(w^{-1}x\sigma(w))=\sigma^{-1}(v_2\sigma(w))w^{-1}v_1=\eta_G(x)$, implying the first equality. For the second equality notice that the length function on $W_{M_J}$ is the restriction of the length function for $W_G$. Thus it is enough to show that $\eta_G(\tilde x)=\eta_{M_J}(\tilde x).$ 

Recall that $\tilde x=w^{-1}x\sigma(w)\in \widetilde W_J.$ The elements $\eta_G(\tilde x)$ resp.~$\eta_{M_J}(\tilde x)$ are computed via the unique decompositions $\tilde x=(w^{-1}v_1)t^{\mu}(v_2\sigma(w))$ and $\tilde x=\tilde w_1 t^{\mu_J}\tilde w_2$ where $w^{-1}v_1,v_2\sigma(w)\in W={}^{\mu}W$, $\mu$ is $G$-dominant, $\tilde w_i\in W_J$ and $\mu_J$ is $M_J$-dominant. We want to show that already these two decompositions coincide, i.e. that $\mu_J$ is $G$-dominant. For this it is enough to prove that if we write $\tilde x=t^{\tilde \mu}\tilde w$ with $\tilde w\in W_J$, then $\langle \beta,\tilde \mu\rangle \geq 0$ for all $\beta\in \Phi^+- \Phi^+_J$. 

Conjugation by $\tilde x$ (or rather by its image $\tilde w$ in the finite Weyl group) induces a permutation of $\Phi^+- \Phi^+_J$. Let $\alpha= w(\beta)\in \Phi$. Since $x$ is a $(J,w,\delta)$-alcove element and shrunken, we obtain $U_{\alpha}\cap {}^xI\subsetneq U_{\alpha}\cap I$. Conjugation by $w^{-1}$ yields 
\begin{equation}\label{eqlemlast}U_{\beta}\cap {}^{\tilde x\delta(w^{-1})}I\subsetneq U_{\beta}\cap {}^{w^{-1}}I.
\end{equation}
 Normalizing the $U_{\gamma}$ such that $U_{\gamma}\cap I\subseteq U_{\gamma}(\O_L)$ with equality if and only if $\gamma$ is negative, we obtain that the right hand side is contained in $U_{\beta}(\O_L)$. By the strict inclusion, the left hand side is contained in $U_{\beta}(t\O_L)$.  We write $U_{\beta}\cap {}^{\tilde x\delta(w^{-1})}I={}^{\tilde\mu(t)}(U_{\beta}\cap {}^{\tilde w\delta(w^{-1})}I)$. We have $U_{\beta}\cap {}^{\tilde w\delta(w^{-1})}I\supseteq U_{\beta}(t\O_L)$. Hence by \eqref{eqlemlast}, $\langle\beta,\tilde\mu\rangle\geq 0$ which is what we wanted.
\end{proof}

Altogether we have

\begin{thm} \label{thmend}
Let $x\a$ be a $(J,w,\delta)$-alcove in the regular shrunken Weyl chambers. Let $[b]_G\in B(G)_x$. Then for every $[b]_{M_J}\in B(M_J)_{\tilde x}$ mapping to $[b]_G$ and such that $\nu_b^{M_J}=\nu_b^G$ we have
$\Delta^G_x(b)= \Delta^{M_J}_{\tilde x}(b).$
\end{thm}
\begin{proof}
We have
\begin{itemize}
\item Proposition \ref{propprop(1)} to compare the first and fourth summand in \eqref{eq5} to the analogous summands for $M_J$ instead of $G$. The proposition gives the desired comparison under the additional assumption that the $M_J$-dominant Newton point $\nu_b$ of $b$ is equal to its $G$-dominant Newton point $\nu_b^G$ (which by (2) holds if $\nu_b$ is basic in $M_J$).
\item $\ell_G(\eta_G(x))=\ell_{M_J}(\eta_{M_J}(w^{-1}x\sigma(w)))$ for $x$ regular shrunken by Lemma \ref{lemlast}.
\item $\defect_G(b)=\defect_{M_J}(b)$ for all $b$ and all $x$, by definition of the defect.
\item $\dim(X^G_x(b))= \dim(X^{M_J}_{\tilde x}(b))$ by \cite{GHKR2}, Theorem 1.1.4(b).
\end{itemize}
Together with \eqref{eq5}, this implies the theorem.
\end{proof}
\begin{proof}[Proof of Theorem \ref{thmmain}]
We apply the above theorem to the case where $J,w$, and $[b]$ are chosen such that $[b]_{M_J}$ is basic in $M_J$. By \eqref{eqbasic} we have $\Delta^{M_J}_{\tilde x}(b)=0$. The theorem follows.
\end{proof}

\section*{Appendix - Comparing closure relations of Newton strata}
\setcounter{section}{8}
The aim of this appendix is to complement the other results of the paper by proving the following more general geometric version of Corollary \ref{cor331}. It gives a direct means to compare closures of Newton strata for Iwahori double cosets for $G$ and a Levi subgroup. We use the same notation as in the other sections. Recall from Construction \ref{remtransform} that if $x$ is normalized, then $I_MxI_M= wI_J\tilde x I_J\sigma(w^{-1})$.

\begin{thm}\label{thmorder}
Let $x\in \widetilde{W}$ such that $x\a$ is a normalized $(J,w,\delta)$-alcove. Let $b,b'\in I_J \tilde xI_J$. Then $b\in \overline{[b']_G\cap I_J\tilde xI_J}$ if and only if $wb\sigma(w^{-1})\in \overline{[b']_G\cap IxI}$ where the closure is taken in the respective Schubert cell. More precisely, assume that $wb\sigma(w^{-1})\in IxI$ is a specialization of a point $\tilde b_{\eta}\in ([b']_G\cap IxI)(\kappa(\tilde b_{\eta}))$ defined over some field $\kappa(\tilde b_{\eta})$, and let $\tilde C$ be the $I$-$\sigma$-conjugacy class of $\tilde b_{\eta}$. Then the closure of $\tilde C\cap I_MxI_M$ contains $wb\sigma(w^{-1})$.
\end{thm}

Our proof is a refinement of the proof of \cite{GoertzHeNie}, Theorem 3.3.1, or \cite{GHKR2}, Theorem 2.1.2. A technical ingredient is the following lemma.

\begin{lemma}\label{lemvector}
Let $(R,\m)$ be a complete discrete valuation ring with algebraically closed residue field. Let $M\in \Mat_{n\times n}(R)$ and $v\in R^n$.
\begin{enumerate}
\item There is a $w\in R^n$ with $w-M\sigma(w)=v$.
\item There is a ring $(R',\m')$ satisfying the same assumptions as $R$, a surjective finite flat morphism $\Spec R'\rightarrow \Spec R$ and a $w\in \mathbb (R')^n$ with $\sigma(w)-Mw=v$.

\item If in addition $v\equiv 0\pmod{\m}$ in (1) or (2), then we may choose $w\equiv 0$ modulo $\m$ resp.~$\m'$.
\end{enumerate}
\end{lemma}
\begin{proof}
For the first statement, \cite{GHKR2}, Lemma 5.1.1 implies that the equation has a solution modulo $\m$. Then we use induction on $i$ to show that the solution can be lifted to a solution modulo $\mathfrak m^i$. Indeed, if $w_i\in\mathbb A^n(R)$ is a solution modulo $\m^i$, then for $d\equiv 0\pmod{\m^i}$ we have $(w_i+d)-M\sigma(w_i+d)\equiv (w_i+d)-M\sigma(w_i)\pmod{\m^{qi}}$, thus $w_{qi}=w_i+ (v-w_i+M\sigma(w_i))$ is a solution modulo $\m^{qi}$. Since $R$ is complete, this proves (1), and also (3) for this case.

For (2) let $\tilde R=R[w_1,\dotsc, w_n]/(w^q-Mw-v)$, a finite integral extension of $R$. Let $\mathfrak p$ be any minimal prime ideal of $\tilde R$. Then $\mathfrak p\cap R=\{0\}$, hence $R\hookrightarrow \tilde R/\mathfrak p$ is a finite integral extension, and $\tilde R/\mathfrak p$ is again an integral domain. Let $R'$ be the integral closure of $\tilde R/\mathfrak p$ in its fraction field (which is a finite field extension of $\Quot R$). Then since $R$ is a complete discrete valuation ring, $R'$ has all claimed properties of (2). For (3) we choose $\mathfrak p$ to be contained in $(\m, w_1,\dotsc, w_n)$ which is a maximal ideal since $v\in \m $.
\end{proof}

\begin{proof}[Proof of Theorem \ref{thmorder}]
By $\sigma$-conjugation with $w$, the condition $b\in \overline{[b']_{M_J}\cap I_J\tilde xI_J}$ is equivalent to $wb\sigma(w^{-1})\in \overline{w[b']_{M_J}\sigma(w^{-1})\cap I_Mx\sigma(I_M)}$. Since the latter set is contained in $\overline{[b']_G\cap IxI}$, one implication follows. 

The $I$-$\sigma$-conjugacy class $\tilde C$ is admissible in $G$ (i.e.~invariant under right multiplication by some open subgroup of $I$, see for example \cite{VWu}, Prop. 2.3), and $I_M x\sigma(I_M)$ and the $I_M$-$\sigma$-conjugacy class of $wb\sigma(w^{-1})$ are admissible in ${}^wM_J=M$. Thus for the other direction it is enough to show that for all elements $I'$ of a fundamental system of neighborhoods of 1 consisting of subgroups of $I$, the image of $c:=wb\sigma(w^{-1})$ in $G(L)/I'$ is contained in the closure of the intersection of $[b']_{G}I'/I'$ and $I_Mx\sigma(I_M)I'/I'$. Furthermore, the assertion holds for $wb\sigma(w^{-1})$ if and only if it holds for some $I_M$-$\sigma$-conjugate of $wb\sigma(w^{-1})$. Thus we may assume that $wb\sigma(w^{-1})\in I_M x$.

By our assumption there is an element $\tilde b\in (IxI)(R)\subseteq LG(R)$ for some complete discrete valuation ring $(R,\m)$ with unformizer $\varepsilon$ and algebraically closed residue field $k$ such that its generic point $\tilde b_{\eta}$ is in $[b']_G\cap IxI$ and such that its special point satisfies $\tilde b_{0}=wb\sigma(w^{-1})$. By \cite{irrmin}, Remark 5.7 we can decompose $\tilde b= i_1xi_2$ with $i_1,i_2\in I(R)$. Here, $i_2$ is unique up to left multiplication by elements of $I\cap x^{-1}Ix$. For the reductions modulo $\m$ we have $\bar i_1 x\bar i_2\in I_M(k)x$, thus $\bar i_2\in (I\cap x^{-1}Ix)(k)$. We lift $\bar i_2$ to some element $d\in (I\cap x^{-1}Ix)(R)$ and replace $i_2$ by $d^{-1}i_2$ and $i_1$ by $i_1xdx^{-1}$. Thus we may assume that $i_2\equiv 1\pmod{\m}$. Replacing $R$ by $\sigma^{-1}(R)=R(\varepsilon^{1/q})$, we have $\sigma^{-1}(i_2)\in I(R)$. We $\sigma$-conjugate $\tilde b$ by $\sigma^{-1}(i_2)$ (without modifying $\tilde b_0$) and may thus assume that $i_2=1$. Hence it is enough to prove the following claim.

{\it Claim.} There is a fundamental system of neighborhoods of $1$ consisting of subgroups $I'$ of $I$ with the following property. For every $I'$ there is a complete discrete valuation ring $R'$ with algebraically closed residue field $k$, a surjective finite flat morphism $\Spec R'\rightarrow \Spec R$, and an element $i\in I(R')$ with $i\equiv 1\pmod{\mathfrak m_{R'}}$ and $i\tilde b\sigma(i^{-1})\in I_MxI'/I'$.

For the proof of this claim we follow the proof of the surjectivity statement of \cite{GoertzHeNie}, Theorem 3.3.1, which claims a similar assertion for $R$ replaced by an algebraically closed field (so that in their case no extension is needed). 

For $n\in\N$ let $T(L)_n$ be the corresponding congruence subgroup as in \cite{PRag_11}, 2.6. For $r\geq 0$ let $I_r$ be the subgroup of $I$ generated by $T(L)_n$ for $n\geq r$ and all affine root subgroup schemes $H_{\alpha+m}$ with $\alpha\in \Phi$ and $m\geq r$ such that $\alpha+m$ is a positive affine root. Let $I_{r^+}=\bigcup_{s>r}I_s$. Then for all $r>0$, we obtain normal subgroups $I_r$ and $I_{r^+}$ of $I$.

Let $M={}^wM_J$, let $N$ be the subgroup of $G$ generated by the root subgroups for roots in $w(\Phi^+-\Phi^+_J)$ and let $\overline N$ be the subgroup of $G$ generated by the root subgroups for roots in $w(\Phi^--\Phi^-_J)$. By intersecting with $I_r$ resp.~with $I_{r^+}$ we obtain normal subgroups $N_r,\overline{N}_{r}$ resp.~$N_{r^+},\overline{N}_{r^+}$. The condition that $x\a$ is a $(J,w,\delta)$-alcove then implies that ${}^{x\sigma}N_r\subseteq N_r$ and ${}^{x\sigma}\overline{N}_r\supseteq\overline{ N}_r$.

\begin{lemma}\label{lemghn341} 
Let $(R,\m)$ be a complete discrete valuation ring with algebraically closed residue field $k$. Let $m\in I_M(R)$ and $r\geq 0$. 

\begin{enumerate}
\item For every $i_+\in N_r(R)$ there is a $b_+\in N_r(R)$ with $b_+i_+{}^{mx\sigma}b_+^{-1}\in N_{r^+}(R)$.
\item For every $i_-\in \sigma(\overline{N}_r(R))$ there is a ring $(R',\m')$ satisfying the same assumptions as $R$, a surjective finite flat morphism $\Spec R'\rightarrow \Spec R$ and an element $b_-\in \overline{N}_r(R')$ such that ${}^{(mx)^{-1}}b_-i_-{}^{\sigma}b_-^{-1}\in \sigma(\overline{N}_{r^+})$.
\item If $i_-$ resp.~$i_+$ are congruent to $1$ modulo $\m$, then $b_-$ resp.~$b_+$ can be chosen to be congruent to $1$ modulo $\m$ resp.~$\m'$.
\end{enumerate}
\end{lemma}
\begin{proof}
The proof of (1) and (2) is the same as that of \cite{GoertzHeNie}, Lemma 3.4.1, replacing references to \cite{GHKR2}, Lemma 5.1.1 by Lemma \ref{lemvector} above. Assertion (3) follows in the same way using (3) of the lemma.
\end{proof}
We return to the proof of the theorem.

We consider a generic Moy-Prasad filtration as explained in \cite{GHKR2}, Section 6. This yields a filtration $I=\bigcup_{r\geq 0}I[r]$ with $I[r]\supset I[s]$ for $r<s$ such that each $I[r]$ is normal in $I$ and a semidirect product of $I[r^+]=\bigcup_{s>r} I[s]$ and some $I\langle r\rangle$ which is either an affine root subgroup or contained in $T(L)_0$. In particular, we also get a corresponding decomposition for every $R$-valued point of $I[r]$.

By the same argument as in \cite{GHKR2}, Section 6 one shows that for every $i$, there is an extension $R'$ of $R$ as in the claim, and an $h\in ({}^{\delta^{-1}(x^{-1})}I\cap I)(R')$ with $h\equiv 1\pmod{\m'}$ such that $h\tilde b\sigma(h^{-1})\in I[i^+]I_Mx$. Since $I_Mx$ is bounded, the subgroups $x^{-1}I_MI[i^+]I_Mx$ form a fundamental system of neighborhoods of $1$ as in the claim. 
\end{proof}

\bibliographystyle{alpha}
\bibliography{references}

\end{document}